\documentclass{article}

\usepackage[all]{xy}
\usepackage{amsmath,amssymb}
\usepackage{amsthm}
\usepackage{enumerate}

\theoremstyle{plain}
  \newtheorem{theorem}{Theorem}[section]
  \newtheorem{corollary}[theorem]{Corollary}
  \newtheorem{lemma}[theorem]{Lemma}
  
  \newtheorem{proposition}[theorem]{Proposition}

\theoremstyle{definition}
  \newtheorem{definition}[theorem]{Definition}

\newcommand{\op}{\mathrm{op}}
\newcommand{\cat}{\mathbf{Cat}}

\newcommand{\ob}{\mathrm{ob}}
\newcommand{\id}{\mathrm{id}}
\newcommand{\hocolim}{\mathrm{hocolim \,}}

\newcommand{\ra}{\longrightarrow}

\newcommand{\gr}{\mathrm{Gr}}
\newcommand{\D}{\mathcal{D}}
\newcommand{\U}{\mathcal{U}}

\renewcommand{\Bbb}{\mathbb}

\author{Kohei Tanaka}

\title {\textbf{\v{C}ech complexes for covers of small categories}}

\begin{document}

\renewcommand{\thesection}{\arabic{section}}

\maketitle

\begin{abstract}
We present a combinatorial analogue of the nerve theorem for covers of small categories, using the Grothendieck construction. 
We apply our result to prove the inclusion-exclusion principle for the Euler characteristic of a finite category.
\end{abstract}





\section{Introduction}

For an open cover $\U=\{U_{a}\}_{a \in A}$ of a space $X$, the \v{C}ech nerve $\check{C}(\U)$ of $\U$ is a 
simplicial space defined by
\[
\check{C}(\U)_{n} = \coprod_{a_{0}, \ldots, a_{n} \in A} U_{a_{0} \ldots a_{n}},
\]
where $U_{a_{0} \ldots a_{n}}$ denotes the intersection $U_{a_{0}} \cap \ldots \cap U_{a_{n}}$.
Here, the face maps omit indices, and the degeneracy maps insert copies of indices.
Dugger and Isaksen have proved that the natural map $\hocolim \check{C}(\U) \to X$ from the homotopy colimit of the \v{C}ech nerve to the original space is a weak homotopy equivalence \cite{DI04}.
In this paper, a combinatorial analogue of their result is provided for small categories.
Let $C$ be a small category, and let $\D=\{D_{a}\}_{a \in A}$ be a collection of subcategories of $C$ satisfying $C= \bigcup_{a \in A}D_{a}$.
We can define the {\em \v{C}ech nerve} $\check{C}(\D)$ of $\D$ as a simplicial object in the category of small categories, similarly to the above.
In the context of small categories, the Grothendieck construction is known as a model of the homotopy colimit of a diagram in small categories \cite{Tho79}.
The Grothendieck construction $\gr(\check{C}(\D))$ of the \v{C}ech nerve is also equipped with the natural functor $\gr(\check{C}(\D)) \to C$.
A natural question to ask is when this induces a homotopy equivalence on classifying spaces.

Unfortunately, it does not hold for every cover. For example, let $C$ be a small category, formed as follows:
\[
\xymatrix{
x \ar@/^/[r] \ar@/_/[r]& y \ar[r] & z,
}
\]
with the terminal object $z$. If $D_{1}$ is the full subcategory with $\ob(D_{1})=\{x,y\}$, and $D_{2}$ is the full subcategory with $\ob(D_{2})=\{y,z\}$, then $C=D_{1} \cup D_{2}$. 
The classifying space of $C$ is contractible; however, the classifying space of the Grothendieck construction for the \v{C}ech nerve of the cover $\{D_{1},D_{2}\}$ is homotopy equivalent to a circle.

In order to solve the problem above, we introduce two classes of covers of small categories consisting of {\em ideals} and {\em filters}, respectively.
These are generalized notions of ideals and filters of posets \cite{Sta12}, \cite{Zap98}.
An ideal $D$ of a small category $C$ is a full subcategory such that an object $y$ of $C$ belongs to $D$, whenever $C(y,x) \not= \emptyset$ for some object $x$ of $D$.
A cover $\D=\{D_{a}\}_{a \in A}$ of a small category $C$ is called an {\em ideal cover}, if $D_{a}$ is an ideal of $D$ for any $a \in A$. 
We can also define the dual notions; filters and {\em filter covers}.

\begin{theorem}
Let $\D$ be an ideal cover or a filter cover of a small category $C$.
The natural functor $\gr(\check{C}(\D)) \to C$ induces a homotopy equivalence between their classifying spaces.
\end{theorem}

As an application of this result, we will show that we can obtain the inclusion-exclusion principle of the Euler characteristics for finite categories with ideal covers or filter covers.
The Euler characteristic $\chi(C)$ of a finite category $C$ is introduced by Leinster in \cite{Lei08}.
He calculates the Euler characteristic of the Grothendieck construction for a diagram in finite categories.
By applying this to the (reduced) \v{C}ech nerve of a cover for a finite category, we obtain the following formula.

\begin{theorem}[Inclusion-exclusion principle]
Let $\D=\{D_{a}\}_{a \in A}$ be an ideal cover or a filter cover of a finite category $C$, indexed by a totally ordered finite set $A$. 
If each intersection $D_{a_{0} \ldots a_{i}}$ and $C$ have Euler characteristics, then we have
\[
\chi(C) = \sum_{i=0}\sum_{a_{0}< \ldots < a_{i}} (-1)^{i} \chi(D_{a_{0} \ldots a_{i}}).
\]
\end{theorem}

Indeed, Leinster has demonstrated the inclusion-exclusion principle of the Euler characteristic for finite sets, as Example 3.4 (d) in \cite{Lei08}.
The theorem above is a generalization of that result.

The rest of this paper is organized as follows.
Section 2 describes the homotopy theory of \v{C}ech complexes, for covers of small categories.
Subsequently, we focus on the inclusion-exclusion principle for Euler characteristics of finite categories in section 3.

\section{The \v{C}ech nerve for covers of small categories}

Let $C$ be a small category, and let $\D=\{D_{a}\}_{a \in A}$ be a collection of subcategories of $C$ indexed by a set $A$.
The intersection $\bigcap_{a \in A} D_{a}$ is a subcategory of $C$, whose set of objects is $\bigcap_{a \in A} \ob(D_{a})$ 
and set of morphisms is $\bigcap_{a \in A} D_{a}(x,y)$, for each pair of objects $x$ and $y$.

On the other hand, the union $\bigcup_{a \in A} \D_{a}$ is a subcategory of $C$, whose set of objects is $\bigcup_{a \in A} \ob(D_{a})$ 
and the set of morphisms is generated by compositions of morphisms of $D_{a}$. That is: 
\[
\left(\bigcup_{a \in A} D_{a} \right)(x,y)=\{f_{n} \ldots f_{0} \mid f_{i} \in D_{a_{i}}(z_{i},z_{i+1}), z_{0}=x, z_{n+1}=y\},
\]
for each pair of objects $x$ and $y$.

\begin{definition}
Let $C$ be a small category.
A collection of subcategories $\D=\{D_{a}\}_{a \in A}$ of $C$ is called a {\em cover} if $C=\bigcup_{a \in A} D_{a}$.
\end{definition}

\begin{definition}[\v{C}ech nerve]\label{nerve}
Let $\D=\{D_{a}\}_{a \in A}$ be a cover of a small category $C$.
The {\em \v{C}ech nerve} $\check{C}(\D)$ is a simplicial object in the category $\cat$ of small categories, defined as 
\[
\check{C}(\D)_{n} = \coprod_{a_{0}, \ldots,a_{n} \in A} D_{a_{0} \ldots a_{n}},
\]
where $D_{a_{0} \ldots a_{n}}$ denotes the intersection $D_{a_{0}} \cap \cdots \cap D_{a_{n}}$.
Let us denote an object (a morphism) $x$ of $\check{C}(\D)_{n}$ belonging to $D_{a_{0} \ldots a_{n}}$ by $x_{a_{0} \ldots a_{n}}$.
The face map $d_{i} : \check{C}(\D)_{n} \to \check{C}(\D)_{n-1}$ is the functor omitting the index $a_{i}$. That is: 
\[
d_{i}(x_{a_{0} \ldots a_{n}}) = x_{a_{0} \ldots \hat{a_{i}} \ldots a_{n}},
\]
on objects and morphisms. Similarly, the degeneracy map $s_{j} : \check{C}(\D)_{n} \to \check{C}(\D)_{n+1}$ is the functor inserting the index $a_{j}$ between $a_{j}$ and $a_{j+1}$. That is:
\[
s_{j}(x_{a_{0} \ldots a_{n}}) = x_{a_{0} \ldots a_{j}a_{j}a_{j+1} \ldots a_{n}},
\]
on objects and morphisms.

In other words, $\check{C}(\D)$ is a functor from $\Delta^{\op}$ to the category $\cat$, where $\Delta$ consists of totally ordered finite sets $[n]=\{0, \cdots,n\}$ and order preserving maps. 
It takes $[n]$ to $\check{C}(\D)_{n}$, and an order preserving map $\varphi : [m] \to [n]$ in $\Delta$ to the functor
$\varphi_{*} : \check{C}(\D)_{n} \to \check{C}(\D)_{m}$. Here, $\varphi_{*}$ is given by $\varphi_{*}(x_{a_{0} \ldots a_{n}}) =x_{b_{0} \ldots b_{m}}$ on objects and morphisms, satisfying $b_{j}=a_{\varphi(j)}$ for each $j \in \{0, \ldots, m\}$.
\end{definition}

A (pseudo) functor from a small category to $\cat$ provides 
a new small category, called the Grothendieck construction.

\begin{definition}[Grothendieck construction]
Let $C$ be a small category, and let $F : C \to \cat$ be a functor.
The {\em Grothendieck construction} $\gr(F)$ is a small category defined by the following:
\begin{itemize}
\item The set of objects consists of pairs $(c,x)$ of an object $c$ of $C$ and an object $x$ of $F(c)$.
\item For each pair of objects $(c,x)$ and $(d,y)$ of $\gr(F)$, the set of morphisms $\gr(F)((c,x),(d,y))$ consists of pairs of morphisms $(f,g)$, where $f \in C(c,d)$ and $g \in F(d)(F(f)(x),y)$.
\item For each composable pair of morphisms $(f_{1},g_{1}) \in \gr(F)((c,x),(d,y))$ and $(f_{2},g_{2}) \in \gr((d,y),(e,z))$, 
the composition is given by $(f_{2},g_{2})(f_{1},g_{1}) = (f_{2} f_{1}, g_{2}(F(f_{2})(g_{1})))$.
\end{itemize}
\end{definition}

We can take the Grothendieck construction of the \v{C}ech nerve $\check{C}(\D) : \Delta^{\op} \to \cat$ for a cover $\D$ of a small category $C$.
Let us see the details of this category. 
The set of objects of $\gr(\check{C}(\D))$ consists of objects $x_{a_{0} \ldots a_{n}}$, indexed by $a_{0}, \cdots, a_{n} \in A$ and $n \geq 0$.
The set of morphisms $\gr(\check{C}(\D))(x_{a_{0} \ldots a_{n}},y_{b_{0} \ldots b_{m}})$ consists of pairs $(\varphi,f_{b_{0} \ldots b_{m}})$ of an order preserving map $\varphi : [m] \to [n]$ such that $b_{j}=a_{\varphi(j)}$ for each $j$, and a morphism $f_{b_{0} \ldots b_{m}} : x_{b_{0} \ldots b_{m}} \to y_{b_{0} \ldots b_{m}}$ of $D_{b_{0} \ldots b_{m}}$.
It is equipped with the canonical functor $\rho : \gr(\check{C}(\D)) \to C$ that eliminates indices. That is,
$\rho(x_{a_{0} \ldots a_{n}})= x$ on objects, and $\rho(\varphi,f_{a_{0} \ldots a_{n}})=f$ on morphisms.
A natural question is to ask when the functor $\rho$ induces a homotopy equivalence on classifying spaces.
The classifying space $BC$ of a small category $C$ is constructed as the geometric realization of the nerve of $C$ (see \cite{Seg68}, \cite{Qui73}, \cite{Hir03}, for homotopical properties of classifying spaces). For a cover of a small category, the classifying space of the Grothendieck construction for the \v{C}ech nerve is called the {\em \v{C}ech complex}.

\begin{definition}
Let $C$ be a small category. An {\em ideal} $D$ of $C$ is a full subcategory such that an object $y$ of $C$ belongs to $D$, whenever $C(y,x) \not= \emptyset$ for some object $x$ of $D$.
Dually, an {\em filter} $D$ of $C$ is a full subcategory such that an object $y$ of $C$ belongs to $D$, whenever $C(x,y) \not= \emptyset$ for some object $x$ of $D$.
A cover $\D=\{D_{a}\}_{a \in A}$ of $C$ is called an {\em ideal} (a {\em filter}) {\em cover} if $D_{a}$ is an ideal (a filter) for any $a \in A$.
\end{definition}

The notions above are generalizations of ideals and filters of posets \cite{Sta12}, \cite{Zap98}.
Let $D$ be a full subcategory of a small category $C$. The category of complement $C \backslash D$ is defined as the full subcategory whose set of objects is $\ob(C) \backslash \ob(D)$.
If $D$ is an ideal (a filter), then $C \backslash D$ is a filter (an ideal).
In other words, the ideal $D$ yields a functor $C \to \mathcal{P}$, where $\mathcal{P}$ is the poset formed of $0<1$.
The functor sends $\ob(D)$ to $\{0\}$ and $\ob(C \backslash D)$ to $\{1\}$. Conversely, for a functor $F : C \to \mathcal{P}$, the category of fiber $F^{-1}(0)$ is an ideal,
and $F^{-1}(1)$ is a filter of $C$. 

\begin{theorem}\label{rho}
Let $\D=\{D_{a}\}_{a \in A}$ be an ideal cover or a filter cover of a small category $C$.
Then, the natural functor $\rho : \gr(\check{C}(\D)) \to C$ induces a homotopy equivalence between their classifying spaces.
\end{theorem}
\begin{proof}
We first consider the case in which $\D$ is an ideal cover. 
We use Quillen's theorem A \cite{Qui73} for $\rho$.
In order to apply this, we examine the left homotopy fiber $\rho \mathord{\mathord{\downarrow}} x$ of $\rho$, over an object $x$ of $C$.
The set of objects consists of pairs $(y_{b_{0} \ldots b_{n}},f)$ of an object $y_{b_{0} \ldots b_{n}}$ of $\gr(\check{C}(\D))$, and a morphism $f : y \to x$ of $C$.
The set of morphisms $(\rho \mathord{\mathord{\downarrow}} x)((y_{b_{0} \ldots b_{n}},f),(z_{c_{0} \ldots c_{m}},g))$ consists of morphisms 
$(\varphi,h_{c_{0} \ldots c_{m}}) : y_{b_{0} \ldots b_{n}} \to z_{c_{0} \ldots c_{m}}$ of $\gr(\check{C}(\D))$, such that $g \circ h=f$.

On the other hand, we choose and fix an index $a \in A$ such that $x \in \ob(D_{a})$.
The set of objects of the over category $\gr(\check{C}(\D)) \mathord{\downarrow} x_{a}$ consists of morphisms $(v_{j},f_{a}) : y_{a_{0} \ldots a_{n}} \to x_{a}$ of $\gr(\check{C}(\D))$. 
Here, $v_{j}$ is the map $[0] \to [n]$ choosing $j \in \{0, \ldots, n\}$, and $f_{a}$ is a morphism from $y_{a_{j}}=y_{a}$ to $x_{a}$ of $D_{a}$.

The set of morphisms $\left( \gr(\check{C}(\D)) \mathord{\downarrow} x_{a} \right)((v_{j},f_{a}),(v_{k},g_{a}))$ 
consists of morphisms $(\varphi,h)$ of $\gr(\check{C}(\D))$, such that $(v_{k},g_{a}) \circ (\varphi,h) = (v_{j},f_{a})$.

A functor 
\[
F : \rho \mathord{\downarrow} x \ra \gr(\check{C}(\D)) \mathord{\downarrow} x_{a}
\]
is defined by $F(y_{a_{0} \ldots a_{n}},f)=(v_{n+1},f_{a}) : y_{a_{0} \ldots a_{n}a} \to x_{a}$ on objects. 
Note that the object $y$ belongs to $D_{a}$, since $D_{a}$ is an ideal. A morphism $(\varphi,h_{b_{0} \ldots b_{m}}) : (y_{a_{0} \ldots a_{n}},f) \to (z_{b_{0} \ldots b_{m}},g)$ of $\rho \mathord{\downarrow} x$ is sent to a morphism $(\widetilde{\varphi},h_{b_{0} \ldots b_{m}a}) : (v_{n+1},f_{a}) \to (v_{m+1},g_{a})$ of $\gr(\check{C}(\D)) \mathord{\downarrow} x_{a}$, 
where $\tilde{\varphi} : [m+1] \to [n+1]$ is the extension of $\varphi : [m] \to [n]$ with $\tilde{\varphi}(m+1)=n+1$.
Conversely, a functor 
\[
G : \gr(\check{C}(\D)) \mathord{\downarrow} x_{a} \ra \rho \mathord{\downarrow} x
\]
is defined by $G(v_{j},g_{a})=(y_{a_{0} \ldots a_{n}},g)$, for an object $(v_{j},g_{a}) : y_{a_{0} \ldots a_{n}} \to x_{a}$ of $\gr(\check{C}(\D)) \mathord{\downarrow} x_{a}$.
There exists a natural transformation $t : GF \Rightarrow \mathrm{id}$ defined by 
\[
t(y_{a_{0} \ldots a_{n}},f)=(d_{n+1},\id_{a_{0} \ldots a_{n}}) : (y_{a_{0} \ldots a_{n}a},f) \ra (y_{a_{0} \ldots a_{n}},f).
\]
Note that the classifying space of the over category $\gr(\check{C}(\D)) \mathord{\downarrow} x_{a}$ is contractible, since it has a terminal object.
We conclude that the classifying spaces of $\gr(\check{C}(\D)) \mathord{\downarrow} x_{a}$ and $\rho \mathord{\downarrow} x$ are homotopy equivalent and contractible.
Quillen's theorem A completes the proof for the case of ideal covers.
If $\D$ is a filter cover of $C$, the opposite cover $\D^{\op}=\{D_{a}^{\op}\}_{a \in A}$ is an ideal cover of $C^{\op}$.
Hence, the classifying space of the left homotopy fiber $\rho' \mathord{\downarrow} x$ is contractible 
for the natural functor $\rho' : \gr(\check{C}(\D^{\op})) \to C^{\op}$, and any object $x$ of $C^{\op}$.
The left homotopy fiber $\rho' \mathord{\downarrow} x$ of $\rho'$ coincides with the right homotopy fiber $x \mathord{\downarrow} \rho$ of $\rho$ for any object $x \in \ob(C)=\ob(C^{\op})$.
Again, Quillen's theorem A implies the result.
\end{proof}

Thomason's homotopy colimit theorem \cite{Tho79} tells us that the \v{C}ech complex for a cover $\D$ of a small category
is homotopy equivalent to the homotopy colimit 
$\hocolim B\check{C}(\D)$ of the diagram $B\check{C}(\D) : \Delta^{\op} \to \mathbf{Top}$, in the category $\mathbf{Top}$ of spaces.
The diagram $B\check{C}(\D)$ is a Reedy cofibrant simplicial space \cite{Hir03}, with respect to the Str\o m model structure on $\mathbf{Top}$, since all the degeneracy maps are closed cofibrations.
The Bousfield-Kan map $\hocolim B\check{C}(\D) \to |B\check{C}(\D)|$ is a homotopy equivalence by Theorem 18.7.4 of \cite{Hir03}. Consequently, we have a homotopy equivalence
\[
|B\check{C}(\D)| \simeq BC,
\]
for an ideal cover or a filter cover $\D$ of a small category $C$.
Note that the collection $\{BD_{a}\}_{a \in A}$ is a cover consisting of subcomplexes of $BC$. 
In \cite{Bjo03}, Bj\"orner shows the above homotopy equivalence for a simplicial complex with a cover consisting of subcomplexes, where the connectivity of intersections is assumed.
We do not require such connectivity of intersections for the proof of Theorem \ref{rho}.

If the index set $A$ is equipped with a total order, then we can consider the ordered \v{C}ech nerve.

\begin{definition}[Ordered \v{C}ech nerve]
Let $\D=\{D_{a}\}_{a \in A}$ be a cover of a small category $C$, indexed by a totally ordered set $A$.
The {\em ordered \v{C}ech nerve} $\check{C}^{o}(\D)$ is a simplicial object in $\cat$, defined as 
\[
\check{C}^{o}(\D)_{n} = \coprod_{a_{0} \leq \ldots \leq a_{n} } D_{a_{0}a_{1} \ldots a_{n}}.
\]
The face and degeneracy maps are defined similarly to the ordinary \v{C}ech nerve.
\end{definition}

The ordered \v{C}ech nerve is smaller, and often easier to treat, than the ordinary \v{C}ech nerve.
There exists an inclusion $\check{C}^{o}(\D) \to \check{C}(\D)$ of simplicial objects in $\cat$. 
This induces a functor $I : \gr(\check{C}^{o}(\D)) \to \gr(\check{C}(\D))$
between their Grothendieck constructions.

\begin{proposition}\label{I}
Let $\D=\{D_{a}\}_{a \in A}$ be a cover of a small category $C$, indexed by a totally ordered set $A$.
The inclusion functor $I : \gr(\check{C}^{o}(\D)) \to \gr(\check{C}(\D))$ induces a homotopy equivalence between their classifying spaces. 
\end{proposition}
\begin{proof}
For an object $x_{a_{0} \ldots a_{n}}$ of $\gr(\check{C}(\D))$, there exists a permutation $\sigma \in \Sigma_{n}$, such that $a_{\sigma(0)} \leq \ldots \leq a_{\sigma(n)}$ in $A$.
In the case that $a_{i}=a_{j}$ for some $i<j$, we always choose $\sigma(i)<\sigma(j)$.
This yields an inverse functor $J : \gr(\check{C}(\D)) \to \gr(\check{C}^{o}(\D))$, sending $x_{a_{0} \ldots a_{n}}$ to $x_{a_{\sigma(0)} \ldots a_{\sigma(n)}}$ on objects.
For a morphism $(\varphi,f) : x_{a_{0} \ldots a_{n}} \to y_{b_{0} \ldots b_{m}}$ of $\gr(\check{C}(\D))$, 
let $\sigma \in \Sigma_{n}$ and $\tau \in \Sigma_{m}$ be permutations for the canonical reordering of $a_{0} \ldots a_{n}$ and $b_{0} \ldots b_{m}$, respectively.
For simplicity, denote $\alpha_{i}=a_{\sigma(i)}$ and $\beta_{j}=b_{\tau(j)}$, for each $i$ and $j$.
For every $j \in \{0,\cdots,m\}$, we have
\[
\beta_{j} = b_{\tau(j)}=a_{\varphi \tau (j)} = \alpha_{\sigma^{-1} \varphi \tau(j)}.
\]
Consider the composition of maps $J(\varphi)=\sigma^{-1} \varphi \tau : [m] \to [n]$. We can describe
\[
b_{0} \ldots b_{m} = \overbrace{a_{0} \ldots a_{0}}^{\varphi^{-1}(0)} \overbrace{a_{1} \ldots a_{1}}^{\varphi^{-1}(1)} \ldots \overbrace{a_{n} \ldots a_{n}}^{\varphi^{-1}(n)}.
\]
The permutation $\tau$ acts on $b_{0} \ldots b_{m}$ as $\sigma$:
\[
b_{\tau(0)} \ldots b_{\tau(m)} = \overbrace{a_{\sigma(0)} \ldots a_{\sigma(0)}}^{\varphi^{-1}(\sigma(0))} \overbrace{a_{\sigma(1)} \ldots a_{\sigma(1)}}^{\varphi^{-1}(\sigma(1))} \ldots \overbrace{a_{\sigma(n)} \ldots a_{\sigma(n)}}^{\varphi^{-1}(\sigma(n))}.
\]
It follows that $J(\varphi)$ preserves order.
The functor $J$ sends $(\varphi,f)$ to $(J(\varphi),f)$ on morphisms.

Trivially, we have $JI = \mathrm{id}$ on $\gr(\check{C}^{o}(\D))$. Let us examine the composition $IJ$.
A functor $K : \gr(\check{C}(\D)) \to \gr(\check{C}(\D))$ is defined by $K(x_{a_{0} \ldots a_{n}}) = x_{a_{0} \ldots a_{n} a_{\sigma(0)} \cdots a_{\sigma(n)}}$ on objects.
For a morphism $(\varphi,f)$ of $\gr(\check{C}^{o}(\D))$, define $K(\varphi,f) = (\varphi*J(\varphi), f)$, 
where $\varphi*J(\varphi) :[2m+1] \to [2n+1]$ is given as 
\[
\begin{cases}
\varphi*J(\varphi)(i)= \varphi(i) & \mathrm{if}\ 0 \leq i \leq m,\\
\varphi*J(\varphi)(i)= J(\varphi)(i-m-1) &\mathrm{if}\ m+1 \leq i \leq 2m+1.
\end{cases}
\]
There exist natural transformations $t : K \Rightarrow IJ$ with $t(x_{a_{0} \ldots a_{n}}) = (d_{0}^{n},\id_{x})$, and 
$s : K \Rightarrow \mathrm{id}$ with $s(x_{a_{0} \ldots a_{n}}) = (d_{n+1} \ldots d_{2n+1},\id_{x})$, for an object $x_{a_{0} \ldots a_{n}}$ of $\gr(\check{C}(\D))$. 
Therefore, the induced maps from $I$ and $J$ on classifying spaces are homotopy inverses to each other.
\end{proof}

For a cover $\D=\{D_{a}\}_{a \in A}$ of a small category $C$, indexed by a totally ordered set $A$, 
the natural functor $\rho^{o} : \gr(\check{C}^{o}(\D)) \to C$ is given as the composition $\rho^{o}=\rho I$. 
Theorem \ref{rho} and Proposition \ref{I} imply the following.

\begin{corollary}\label{rho^o}
Let $\D=\{D_{a}\}_{a \in A}$ be an ideal cover or a filter of a small category $C$, indexed by a totally ordered set $A$.
The natural functor $\rho^{o} : \gr(\check{C}^{o}(\D)) \to C$ induces a homotopy equivalence between their classifying spaces.
\end{corollary}

Note that the Grothendieck construction of the ordered \v{C}ech nerve is independent of the order on the index set.

\begin{proposition}\label{independent_order}
Let $\D=\{D_{a}\}_{A}$ be a cover of a small category $C$, and let the index set $A$ be equipped with two total order $\leq_{1}$ and $\leq_{2}$.
Denote the ordered \v{C}ech nerve induced from $\leq_{i}$ by $\check{C}^{o}(\D)_{i}$ for $i=1,2$, respectively.
Then, the Grothendieck constructions $\gr(\check{C}^{o}(\D)_{1})$ and $\gr(\check{C}^{o}(\D)_{2})$ are isomorphic to each other as categories.
\end{proposition}
\begin{proof}
Similarly to the proof of Proposition \ref{I}, we obtain a permutation $\sigma$ for an object $x_{a_{0} \ldots a_{n}}$ of $\gr(\check{C}^{o}(\D)_{1})$, such that
$x_{a_{\sigma(0)} \ldots a_{\sigma(n)}}$ belongs to $\gr(\check{C}^{o}(\D)_{2})$. 
This yields a functor $F : \gr(\check{C}^{o}(\D)_{1}) \to \gr(\check{C}^{o}(\D)_{2})$.
We also obtain an inverse functor $G: \gr(\check{C}^{o}(\D)_{2}) \to \gr(\check{C}^{o}(\D)_{1})$ by reordering the indexes.
A straightforward calculation shows that $FG=\mathrm{id}$ and $GF=\mathrm{id}$.
\end{proof}

Next, we define the reduced \v{C}ech nerve, smaller than the ordered \v{C}ech nerve.

\begin{definition}[Reduced \v{C}ech nerve]
Let $\Delta_{inj}$ denote the subcategory of $\Delta$, consisting of the same objects as $\Delta$ and injective order preserving maps.
Let $\D=\{D_{a}\}_{a \in A}$ be a cover of a small category $C$, indexed by a totally ordered set $A$. 
The {\em reduced \v{C}ech nerve} of $\D$ is a functor $\check{C}^{o}(\widetilde{\D}) : \Delta_{inj} \to \cat$, defined by
\[
\check{C}^{o}(\widetilde{\D})_{n} = \coprod_{a_{0} < \ldots < a_{n} } D_{a_{0}a_{1} \ldots a_{n}}.
\]
\end{definition}

This is the diagram in $\cat$ where the degenerate parts of the ordered \v{C}ech nerve have been removed.
It has only face maps omitting indices. 
Note that if $A$ is finite, then $\check{C}^{o}(\widetilde{\D})_{n}$ is the empty category for $n$ greater than the cardinality $A^{\sharp}$ of $A$.

The inclusion $\Delta_{inj} \to \Delta$ induces a functor $L : \gr(\check{C}^{o}(\widetilde{\D})) \to \gr(\check{C}^{o}(\D))$ on Grothendieck constructions.

\begin{lemma}\label{unique}
Let $\D=\{D_{a}\}_{a \in A}$ be a cover of a small category.
For two morphisms $(\varphi,f), (\chi,g): x_{a_{0} \ldots a_{n}} \to y_{b_{0} \ldots b_{m}}$ of $\gr(\check{C}(\D))$, if the indices $a_{0}, \ldots, a_{n}$ are distinguished, then $\varphi = \chi$.
\end{lemma}
\begin{proof}
For any $j \in \{0,\ldots,m\}$, the equality $a_{\varphi(j)}=b_{j}=a_{\chi(j)}$ holds.
Since the indices $a_{0},\ldots,a_{n}$ are distinguished, $\varphi(j)$ and $\chi(j)$ must be equal to each other.
\end{proof}

\begin{proposition}\label{ordered_adjoint}
Let $\D=\{D_{a}\}_{a \in A}$ be a cover of a small category $C$, indexed by a totally ordered set $A$. 
The inclusion functor $L : \gr(\check{C}^{o}(\widetilde{\D})) \to \gr(\check{C}^{o}(\D))$ has a right adjoint functor.
\end{proposition}
\begin{proof}
For an arbitrary object $x_{a_{0} \ldots a_{n}}$ of $\gr(\check{C}^{o}(\D))$, 
we obtain an object $x_{\alpha_{0} \ldots \alpha_{n'}}$ of $\gr(\check{C}^{o}(\widetilde{\D}))$ and 
a surjective order preserving map $\psi : [n] \to [n']$, such that $\alpha_{j}=a_{\psi(j)}$, by removing duplicate indices. 
This determines the reduction functor $R : \gr(\check{C}^{o}(\D)) \to \gr(\check{C}^{o}(\widetilde{\D}))$, sending $x_{a_{0} \ldots a_{n}}$ to $x_{\alpha_{0} \ldots \alpha_{n'}}$ on objects.
For a morphism $(\varphi,f) : x_{a_{0} \ldots a_{n}} \to y_{b_{0} \ldots b_{m}}$ of $\gr(\check{C}^{o}(\widetilde{\D}))$, the order preserving map $\varphi : [m] \to [n]$ induces 
$\varphi' : [m'] \to [n']$, making the following diagram commutative:
\[
\xymatrix{
[m] \ar[r]^{\varphi} \ar[d]_{\psi} & [n] \ar[d]^{\psi} \\
[m'] \ar[r]_{\varphi'} & [n'].
}
\]
The reduction functor $R$ sends $(\varphi,f)$ to $(\varphi',f)$ on morphisms.
We take objects $x_{a_{0} \ldots a_{n}}$ of $\gr(\check{C}^{o}(\widetilde{\D}))$ and $y_{b_{0} \ldots b_{m}}$ of $\gr(\check{C}^{o}(\D))$, respectively.
Since $a_{0}, \ldots, a_{n}$ are distinguished, we have that
\begin{equation}
\begin{split}
\gr(\check{C}^{o}(\widetilde{\D}))(x_{a_{0} \ldots a_{n}}, R(y_{b_{0} \ldots b_{m}})) &=  \gr(\check{C}^{o}(\widetilde{\D}))(x_{a_{0} \ldots a_{n}}, y_{\beta_{0} \ldots \beta_{m'}})) \\ \notag
& \cong D_{\beta_{0} \ldots \beta_{m'}}(x,y) \\ 
&= D_{b_{0} \ldots b_{m}}(x,y) \\
&\cong \gr(\check{C}^{o}(\D))(L(x_{a_{0} \ldots a_{n}}), y_{b_{0} \ldots b_{m}}),
\end{split}
\end{equation}
by Lemma \ref{unique}.
Therefore, we conclude that the reduction functor $R$ is right adjoint to $L$.
\end{proof}

The reduced \v{C}ech nerve is also equipped with the natural functor $\tilde{\rho}=\rho^{o}L : \check{C}^{o}(\widetilde{\D}) \to C$, the same as the ordered \v{C}ech nerve and the ordinary \v{C}ech nerve.
Form Proposition \ref{ordered_adjoint} and Corollary \ref{rho^o}, we can deduce the following.

\begin{corollary}
Let $\D=\{D_{a}\}_{a \in A}$ be an ideal cover or a filter cover of a small category $C$, indexed by a totally ordered set $A$. 
The natural functor $\tilde{\rho} : \check{C}^{o}(\widetilde{\D}) \to C$ induces a homotopy equivalence between their classifying spaces.
\end{corollary}

\begin{definition}
A cover $\D=\{D_{a}\}_{a \in A}$ of a small category $C$ is called {\em locally finite} when
\[
\{a \in A \mid x \in \ob(D_{a})\}^{\sharp}<\infty,
\]
for any object $x$ of $C$. 
\end{definition}

\begin{proposition}\label{finite_adjoint}
Let $\D=\{D_{a}\}_{a \in A}$ be a locally finite ideal cover of a small category $C$, indexed by a totally ordered set $A$.
The natural functor $\tilde{\rho} : \gr (\check{C}^{o}(\widetilde{\D})) \to C$ has a left adjoint functor.
\end{proposition}
\begin{proof}
A functor $\pi : C \to \gr(\check{C}^{o}(\widetilde{\D}))$ is defined by $\pi(x) = x_{a_{0} \ldots a_{n}}$ on objects, where $D_{a_{0}}, \ldots, D_{a_{n}}$ are whole distinguished ideals in $\D$ that contain $x$.
For a morphism $f : x \to y$ of $C$, the object $x$ belongs to an ideal $D$ if $y$ does.
When we describe $\pi(x)=x_{a_{0} \ldots a_{n}}$ and $\pi(y)=y_{b_{0} \ldots b_{m}}$, 
the inclusion relation $\{b_{0}, \ldots b_{m}\} \subset \{a_{0},\ldots, a_{m}\}$ holds.
This yields an injection $\varphi : [m] \to [n]$, such that $b_{j}=a_{\varphi(j)}$.
We define $\pi(f) = (\varphi,f) : \pi(x) \to \pi(y)$ on morphisms.

Take objects $x$ of $C$ and $y_{b_{0} \ldots b_{m}}$ of $\gr(\check{C}^{o}(\widetilde{\D}))$, respectively.
Then, by Lemma \ref{unique},
\[
C(x,\tilde{\rho}(y_{b_{0} \ldots b_{m}}))=C(x,y) =D_{b_{0} \ldots b_{m}}(x,y) \cong \gr(\check{C}^{o}(\widetilde{\D}))(\pi(x),y_{b_{0} \ldots b_{m}}).
\]
Therefore, we conclude that the functor $\pi$ is left adjoint to $\tilde{\rho}$.
\end{proof}

When $\D=\{D_{a}\}_{a \in A}$ is a filter cover of a small category $C$, the opposite cover $\D^{\op}=\{D_{a}^{\op}\}_{a \in A}$ is an ideal cover of $C^{\op}$.
The following corollary then follows immediately. 

\begin{corollary}
Let $\D=\{D_{a}\}_{a \in A}$ be a locally finite filter cover of a small category $C$, indexed by a totally ordered set $A$.
The natural functor $\gr (\check{C}^{o}(\widetilde{\D}^{\op})) \to C^{\op}$ has a left adjoint functor.
\end{corollary}

\section{The inclusion-exclusion principle for the Euler characteristics of finite categories}

In this section, we focus on the inclusion-exclusion principle for the Euler characteristic of a finite category.
The Euler characteristic of a finite category was introduced by Leinster \cite{Lei08}. This is a generalization of M\"{o}bious inversion for posets \cite{Rot64}.
Let us briefly review the definition. In this paper, we use the rational numbers $\mathbb{Q}$ as the value of Euler characteristics of finite categories.

\begin{definition}
Suppose that $C$ is a finite category that has finitely many objects and morphisms.
\begin{enumerate}
\item The {\em similarity matrix} of $C$ is the function 
$\zeta : \mathrm{ob}(C) \times \mathrm{ob}(C) \to \mathbb{Q}$, given by the cardinality of the set of morphisms $\zeta(a,b)=C(a,b)^{\sharp}$.
\item Let $u : \ob(C) \to \mathbb{Q}$ denote the column vector with $u(a)=1$, for any object $a$ of $C$.
A {\em weighting} on $C$ is a column vector $w:\mathrm{ob}(C)\to \mathbb{Q}$ such that $\zeta w=u$. 
Dually, a {\em coweighting} on $C$ is a row vector $v:\mathrm{ob}(C)\to \mathbb{Q}$, such that $v \zeta = u^{*}$, where $u^{*}$ is the transposed matrix of $u$.
\end{enumerate}
\end{definition}

Note that we have
\[
\sum_{i \in \mathrm{ob}(C)}w(i) = u^{*} w = v \zeta w= v u=\sum_{j \in \mathrm{ob}(C)} v(j),
\]
if both a weighting and a coweighting exist. 
Moreover, 
\[
\sum_{i \in  \mathrm{ob}(C)} w(i)= u^{*} w = v \zeta w = v \zeta w' = u^{*} w' = \sum_{i \in \mathrm{ob}(C)} w'(i),
\]
for two (co)weightings $w$ and $w'$ on $C$. This guarantees the following definition.

\begin{definition}
Let $C$ be a finite category. 
We say that $C$ \textit{has Euler characteristic} if it has 
both a weighting $w$, and a coweighting $v$, on $C$. 
Then, the \textit{Euler characteristic} of $C$ is defined by
\[
\chi(C) = \sum_{i \in \mathrm{ob}(C)}w(i) = \sum_{j \in \mathrm{ob}(C)}v(j).
\]
\end{definition}

It is well-known that the topological Euler characteristic $\chi_{T}$ has the following formula:
\[
\chi_{T}(A \cup B) = \chi_{T}(A)+\chi_{T}(B)-\chi_{T}(A \cap B),
\]
for subcomplexes $A$ and $B$ of a finite CW-complex $X$. 
Recall the example given in our introduction. The finite category $C$ is
\[
\xymatrix{
x \ar@/^/[r] \ar@/_/[r]& y \ar[r] & z,
}
\]
with the terminal object $z$, $D_{1}$ is the full subcategory with $\ob(D_{1})=\{x,y\}$, and $D_{2}$ is the full subcategory with $\ob(D_{2})=\{y,z\}$. 
The Euler characteristic $\chi(C)=\chi(D_{1} \cup D_{2})=1$, since $C$ has a terminal object. However, 
\[
\chi(D_{1}) + \chi(D_{2}) -\chi(D_{1} \cap D_{2}) = 0+1-1=0.
\]
The inclusion-exclusion principle does not hold in this case. In \cite{FLS11}, Fiore, Luck, and Sauer used the homotopy colimit of diagram in $\cat$, instead of the genuine colimit or union.

Now, we consider the following situation.
Let $C$ be a finite category, and let $\D=\{D_{a}\}_{a \in A}$ be a cover of $C$, indexed by a totally ordered finite set $A$.
We denote the full subcategory of $\Delta_{inj}$, consisting of $[n]$ for $0 \leq n \leq A^{\sharp}-1$ by $\Delta_{A^{\sharp}}$.
The reduced \v{C}ech nerve $\gr(\check{C}^{o}(\widetilde{\D}))$ can be regarded as a functor $\Delta_{A^{\sharp}} \to \cat$.
The category $\Delta_{A^{\sharp}}$ is a finite acyclic category, which never has circuit of morphisms.
Moreover, each $\check{C}^{o}(\widetilde{\D})_{n}$ is a finite coproduct of finite categories.
Leinster provided the product formula for the Euler characteristic of finite categories, with respect to Grothendieck construction in \cite{Lei08}.
By applying this to $\gr (\check{C}^{o}(\widetilde{\D}))$, we obtain the inclusion-exclusion principle for ideal (filter) covers of a finite category.

\begin{theorem}[Inclusion-exclusion principle]\label{inclusion-exclusion}
Let $\D=\{D_{a}\}_{a \in A}$ be an ideal cover or a filter cover of a finite category $C$, indexed by a finite set $A$. 
If each $D_{a_{0} \ldots a_{i}}$ and $C$ have Euler characteristics, then we have
\[
\chi(C) = \sum_{i=0}^{A^{\sharp}-1}\sum_{a_{0}< \ldots < a_{i}} (-1)^{i} \chi(D_{a_{0} \ldots a_{i}}).
\]
\end{theorem}
\begin{proof}
Let $\D$ be an ideal cover of $C$.
The category $\Delta_{A^{\sharp}}$ has a unique coweighting $v[n]=(-1)^{n}$ (see Example 3.4 (d) of \cite{Lei08}).
By applying Proposition 2.8 of \cite{Lei08} to $\gr (\check{C}^{o}(\widetilde{\D}))$, we have
\[
\chi (\gr(\check{C}^{o}(\widetilde{\D})))= \sum_{i=0}^{A^{\sharp}-1} (-1)^{i}\chi(\check{C}^{o}(\widetilde{\D})_{i})=\sum_{i=0}^{A^{\sharp}-1}\sum_{a_{0}< \ldots <a_{i}}(-1)^{i}\chi(D_{a_{0} \ldots a_{i}}).
\] 
Then, Proposition 2.4 of \cite{Lei08} and our Proposition \ref{finite_adjoint} imply the result.
If $\D$ is a filter cover of $C$, then $\D^{\op}$ is an ideal cover of $C^{\op}$. 
Since the equality $\chi(X)=\chi(X^{\op})$ holds for any finite category $X$ having Euler characteristic, we obtain the desired formula.
\end{proof}

Note that the order on the index set is not essential, since we can choose a total order on any finite set, and Proposition \ref{independent_order} holds in the case of reduced \v{C}ech nerves too.

\begin{corollary}
Let both $A$ and $B$ be filters or ideals of a finite category $C$. If each of $A$, $B$, $A \cap B$, and $A \cup B$ has Euler characteristic, then we have
\[
\chi(A \cup B)=\chi(A)+\chi(B)-\chi(A \cap B).
\]
\end{corollary}

Institute of Social Sciences, School of Humanities and Social Sciences, Academic Assembly, Shinshu University, Japan.

\textit{E-mail address}: tanaka@shinshu-u.ac.jp

\end{document}